\documentclass[reqno,12pt]{amsart}
\newcommand{\be}{\begin{eqnarray}}
\newcommand{\ee}{\end{eqnarray}}
\newcommand{\beq}{\begin{equation}}
\newcommand{\eeq}{\end{equation}}
\newcommand{\ben}{\begin{eqnarray*}}
\newcommand{\een}{\end{eqnarray*}}

\newtheorem{theorem}{Theorem}

\newtheorem{defi}{Definition}

\newtheorem{lemma}[theorem]{Lemma}

\newtheorem{proposition}[theorem]{Proposition}
\newtheorem{remark}[theorem]{Remark}

{\catcode`\@=11\global\let\AddToReset=\@addtoreset
\AddToReset{equation}{section}

\AddToReset{theorem}{section}

\usepackage{graphicx}
\usepackage{psfrag}
\usepackage{color}

\newcommand{\RR}{\mathbb R}

\newcommand{\rbarbetad}{r_{\bar\beta}}

\begin{document}
\title[Multiple existence of sign changing solutions]
      {Multiplicity results for sign changing bound state solutions of
            a semilinear equation   }\thanks{This research was supported by
        FONDECYT-1110074 for the first author,
        FONDECYT-1110268  for the second author and FONDECYT-11121125 for the  third author.}
\author{Carmen Cort\'azar}
\address{Departamento de Matem\'atica, Pontificia
        Universidad Cat\'olica de Chile,
        Casilla 306, Correo 22,
        Santiago, Chile.}
\email{\tt ccortaza@mat.puc.cl}
\author{Marta Garc\'{\i}a-Huidobro}
\address{Departamento de Matem\'atica, Pontificia
        Universidad Cat\'olica de Chile,
        Casilla 306, Correo 22,
        Santiago, Chile.}
\email{\tt mgarcia@mat.puc.cl}
\author{Pilar Herreros}
\address{Departamento de Matem\'atica, Pontificia
        Universidad Cat\'olica de Chile,
        Casilla 306, Correo 22,
        Santiago, Chile.}
\email{\tt pherrero@mat.puc.cl}


\begin{abstract}
In this paper we give conditions on $f$ so that problem
 $$ \Delta u
+f(u)=0,\quad x\in \RR^N, N\ge 2, $$
 has at least two radial bound state solutions with any prescribed number of zeros, and such that $u(0)$ belongs to a specific subinterval of $(0,\infty)$. This property will allow us to give conditions on $f$ so that this problem has at least any given number of radial solutions having a  prescribed number of zeros.

\end{abstract}

\maketitle

\section{Introduction and main results}

In this paper we  give conditions on the nonlinearity $f$ so that the problem
\begin{eqnarray}\label{pde}
\begin{gathered}
 \Delta u
+f(u)=0,\quad x\in \RR^N, N\ge 2, \\
 \lim\limits_{|x|\to\infty}u(x)=0,
\end{gathered}
\end{eqnarray}
 has at least two solutions with $u(0)>0$ having any  prescribed number of nodal regions. To this end we consider the radial version of \eqref{pde}, that is
\begin{eqnarray}\label{eq2}
\begin{gathered}
u''+\frac{N-1}{r}u'+f(u)=0,\quad r>0,\quad N\ge 2,\\
u'(0)=0,\quad \lim\limits_{r\to\infty}u(r)=0,
\end{gathered}
\end{eqnarray}
where all throughout this article  $'$ denotes differentiation with respect to $r$.

Any nonconstant solution to \eqref{pde} is called a bound state solution. Bound state solutions such that $u(x)>0$ for all $x\in\mathbb R^N$, are referred to as a first bound state solution, or  a ground state solution.

\medskip

The existence of  solutions for \eqref{pde} has been established by many authors under different regularity and growth assumptions on the nonlinearity $f$. For the existence of ground state solutions  see for example \cite{b-l1, fg, fls, gst} and the references therein.   The existence of infinitely many  radial bound states was first proved in \cite{stra} and then generalized in \cite{b-l2}. Later,   \cite{cdghm, cghy4, dghm,  jk, mtw} proved the existence of at least one solution of \eqref{eq2} with $u(0)>0$  having any prescribed number of zeros. For the non-autonomous case we refer to \cite{bw1, cmt, stru} and for the non-radial case we refer to \cite{bw2, cds, cmp, mpw} and the references therein.
\medskip

The uniqueness problem for positive solutions to problem \eqref{pde} has been extensively studied during the past decades, see for example  \cite{fls, ms, pel-ser2, pu-ser, st}. More recently, some results concerning the uniqueness of higher order bound states have been obtained, see  \cite{troy,cghy,cghy2}.

As for multiplicity results, the following non-autonomous problem
$$-\Delta u=f(x,u),\qquad u(x)\to0\quad\mbox{as $|x|\to\infty$}$$
has been considered for a strictly non-autonomous $f$ of the form $f(x,u)=g(x,u)-a(x)u$ by \cite{at1, at2, aw,  cmp, cps1, cao, dwy, hl, lu, zhu}. Under different assumptions on the nonnegative function $g$ and the coefficient $a$, they have established existence of multiple ground state solutions.

\medskip
In this paper we study the autonomous case. We give conditions on $f$ so that problem \eqref{eq2} has at least two solutions with any prescribed number of zeros, and such that $u(0)$ belongs to a specific subinterval of $(0,\infty)$. This property will allow us to give conditions on $f$ so that problem \eqref{eq2} has at least any given number of solutions having a  prescribed number of nodes.

We will work under the following two sets of assumptions on the nonlinearity $f$:\\

$\mathbf{(A1)}$    Finite case:
$\gamma_*<\infty$
\begin{enumerate}
\item[$(f_1)$] $f$ is a continuous  function defined in $(\gamma_*^-,\gamma_*]$, $-\infty\le\gamma_*^-<0<\gamma_*$, $f(0)=0$, $f(\gamma_*)=0$, and $f$ is locally Lipschitz in $(\gamma_*^-,\gamma_*]\setminus\{0\}$.
\item[$(f_2)$]
There exists $\delta>0$ such that if we set $F(s)=\int_0^sf(t)dt$, it holds that $F(s)<0$ for all $0<|s|<\delta$, and $\lim_{s\to\gamma_*^-}F(s)=F(\gamma_*)$, $F(s)<F(\gamma_*)$ for all $s\in(\gamma_*^-,\gamma_*)$.
\item[$(f_3)$] $F$ has a local maximum at some $\gamma\in(\delta,\gamma_*)$ with $F(\gamma)>0$.
\item[$(f_4)$] $f$ has a finite number of zeros in $(\gamma_*^-,-\delta)\cup(\delta,\gamma_*)$ and $f$ changes sign at these  points.
\end{enumerate}
\medskip

$\mathbf{(A2)}$    Infinite case: ($\gamma_*=\infty$)
\begin{enumerate}
\item[$(f_1)$] $f$ is a continuous  function defined in $(\gamma_*^-,\infty)$, $-\infty\le\gamma_*^-<0$, $f(0)=0$  and $f$ is locally Lipschitz in $(\gamma_*^-,\infty)\setminus\{0\}$.
\item[$(f_2)$]
There exists $\delta>0$ such that if we set $F(s)=\int_0^sf(t)dt$, it holds that $F(s)<0$ for all $0<|s|<\delta$, and $F(s)<\lim\limits_{s\to\infty}F(s)= \lim\limits_{s\to\gamma_*^-}F(s)$ for all $s$.
\item[$(f_3)$] $F$ has a local maximum at some $\gamma\in(\delta,\infty)$ with $F(\gamma)>0$.
\item[$(f_4)$] $f$ has a finite number of zeros in $(\gamma_*^-,-\delta)\cup(\delta,\infty)$ and $f$ changes sign at these  points.
\item[$(f_5)$]
 There exists $s_0\in(\gamma_*^-,0)$  such that
 $Q(s)>0$ for all $s\in(\gamma_*^-,s_0)$, and there exists $\theta\in(0,1)$
\ben\lim_{s\to\infty}\Bigl(\inf_{s_1,s_2\in[\theta s,s]}Q(s_2)\Bigl(\frac{s}{f(s_1)}\Bigr)^{N/2}\Bigr)=\infty,
\een
where $Q(s):=2NF(s)-(N-2)sf(s)$.
\end{enumerate}

\medskip

As the Lipschitz assumption on $f$ in $(f_1)$ does not include $\{0\}$, the solutions that we obtain may have compact support, see for example \cite{fls}.

In order to state our results, we define some constants that will be used throughout this paper:
\begin{defi}\label{constants}\rm Under assumptions $(A1)$ or $(A2)$, we define the following special constants:
\begin{enumerate}
\item[(i)] We set $\gamma_0=0$, and denote by $\gamma_1$ the first positive local maximum point for $F$ such that $F(\gamma_1)>0$. Next, for $i\in\mathbb N$, we denote by $\gamma_{i+1}$ the first maximum point of $F$ occurring after $\gamma_i$ such that $F(\gamma_i)< F(\gamma_{i+1})$, with the convention that the last one is $\gamma_M$ and we set  $\gamma_{M+1}=\gamma_*$. Similarly, we denote by $\gamma_{-1}$ the first local  negative maximum point (if any) for $F$  with $F(\gamma_{-1})>0$ and we denote by $\gamma_{i-1}$ the first local maximum of $F$ which occurs to the left of $\gamma_{i}$ such that $F(\gamma_{i})< F(\gamma_{i-1})$ with the convention that the last one is $\gamma_{\bar M}$ and we set $\gamma_{\bar M-1}=\gamma_*^-$. If there are no negative local   maximum points for $F$ with $F>0$, we will define $\bar M=0$ and  $\gamma_{\bar M-1}=\gamma_{-1}=\gamma_*^-$.

\item[(ii)]For $i\ge 1$, we  denote by $\beta_i$ the largest point in $(\gamma_{i-1},\gamma_i)$ such that $F(\beta_i)=F(\gamma_{i-1})$ and denote by $\beta_*$ the largest point in $(\gamma_M,\gamma_*)$ (or in $(\gamma_M,\infty)$) where $F(\gamma_M)=F(\beta_*)$. Similarly, for $i\le -1$, we define $\beta_{i}$ as the smallest point in $(\gamma_{i},\gamma_{i+1})$ such that $F(\beta_{i})=F(\gamma_{i+1})$, and  $\beta_*^-$ as the smallest point in $(\gamma_*^-,\gamma_{\bar M})$ where $F(\gamma_{\bar M})=F(\beta_*^-)$.
 \end{enumerate}
\mbox{ }\quad Finally, we identify a positive constant $\bar\beta$ as follows:
\begin{enumerate}

\item[(iii)]If $f$ satisfies $(A1)$, we choose $\bar\beta>\beta_*$ such that $F(\bar\beta)>F(\beta_*^-)$ and if $f$ satisfies $(A2)$ we define $\bar\beta$ as a point $\bar\beta>\beta_*$, such that $F(\bar\beta)>F(\beta_*^-)$ and $Q(s)>0$ for all $s$ satisfying $F(s)>F(\bar\beta)$. (this point exists by $(f_5)$)

 \end{enumerate}
\end{defi}

\begin{figure}[h]
\begin{center}
 \includegraphics[keepaspectratio, width=13cm]{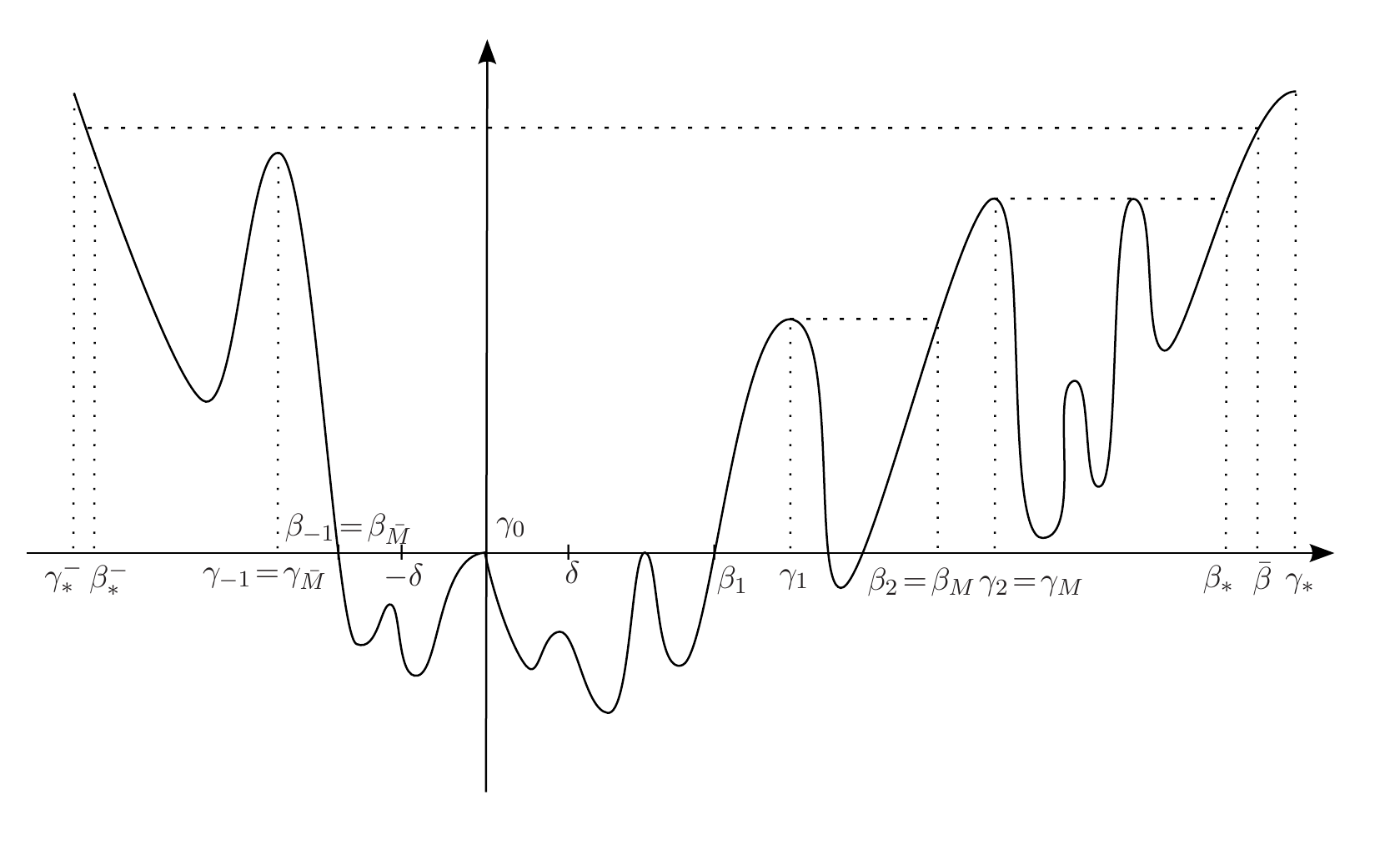}
 \end{center}
 \caption{The function $F$}
 \end{figure}

Our main multiplicity results are the following, where from now on $\gamma_*=\infty$ in the case that $f$ satisfies assumptions $(A2)$.

\begin{theorem}\label{main0}
Assume that $f$ satisfies either assumptions $(A1)$ or $(A2)$.
Then,
there exists $k_0\in\mathbb N\cup\{0\}$ such that for any $k\ge k_0$, there exist at least two solutions $u$ of \eqref{eq2}, with initial value in $(\beta_*,\gamma_*)$, having exactly $k$ sign changes in $(0,\infty)$.

\end{theorem}

Note that for any $i>1$ there exists $\gamma_i^-<0$ such that the restriction of $f$ to the interval $(\gamma_i^-,\gamma_i]$ satisfies condition $(A1)$, and similarly for $i<-1$.
 Also, from the results in \cite{cghy4}, it follows that for any $k\in\mathbb N\cup\{0\}$,  there exists at least one solution $u$ of \eqref{eq2}, with initial value in $(\beta_1,\gamma_1)$, having exactly $k$ sign changes in $(0,\infty)$.
 Hence we immediately obtain the following corollary:
\medskip

\noindent{\bf Corollary.}
Assume that $f$ satisfies either assumptions $(A1)$ or $(A2)$. Then there exists $k_0\in\mathbb N\cup\{0\}$ such that for any $k\ge k_0$, there exist at least $2M+1$ solutions of \eqref{eq2} with a positive initial value, and  at least $2\bar M-1$ solutions of \eqref{eq2} with a negative initial value,  having exactly $k$ sign changes in $(0,\infty)$.
\medskip

Our next result shows that bound states with initial value in $(\beta_*,\gamma_*)$ need not exist for every $k\in\mathbb N\cup\{0\}$:

\begin{theorem}\label{main1}

Let  $\bar u$ denote the largest point in $(\beta_*^-,0)$ such that $F(\bar u)=F(\gamma_1)$. If either $f$ satisfies assumptions $(A1)$ and
\beq\label{t13}-\min_{s\in[\beta_*^-,\beta_{*}]}F(s)<\frac{(\beta_*-\gamma_1)}{2(N-1)(k+1)}\frac{F(\gamma_1)}{\gamma_1-\bar u}-F(\gamma_*),
\eeq
or $f$ satisfies $(A2)$ and
\beq\label{t13b}-\min_{s\in[\beta_*^-,\beta_{*}]}F(s)<\frac{(\beta_*-\gamma_1)}{2(N-1)(k+1)}\frac{F(\gamma_1)}{\gamma_1-\bar u}-\sup_{s\in[0,\alpha_k]}F(s),
\eeq
where $\alpha_k$ is defined in Lemma \ref{nn2},
then there are no solutions $u$ of \eqref{eq2}, with initial value in $(\beta_*,\gamma_*)$, having exactly $j$ sign changes in $(0,\infty)$ for any $j=0,1,\ldots, k$.

\end{theorem}

In our last result we give a sufficient condition so that $k_0=1$ in Theorem \ref{main0}.
In order to state it we define
\ben
\bar F:=-\min_{s\in[0,\beta_1]}F(s)>0.
\een
$$A=\frac{\beta_*-\beta_1}{((F(\bar\beta)-F(\gamma_M)))^{1/2}}+ \Bigl(\frac{2N(\bar\beta-\beta_*)}{\min_{t\in[\beta_*,\bar\beta]}f(t)}\Bigr)^{1/2}\quad\mbox{and}\quad \bar I=F(\gamma_*),$$
if $f$ satisfies $(A1)$, and
$$A=\max\{1, \frac{\beta_*-\beta_1}{((F(\bar\beta)-F(\gamma_M)))^{1/2}}+ \Bigl(\frac{2N(\bar\beta-\beta_*)}{\min_{t\in[\beta_*,\bar\beta]}f(t)}\Bigr)^{1/2}\}$$
and
$$\bar I:=\Bigl(\frac{\bar C+1}{\bar C}\Bigr)^N\Bigl(2F(\bar\beta)+(\bar\beta-\beta_1)^2+\frac{1}{N}\Bigl(\sup_{s\in[\beta_1,\bar\beta]}Q(s)-\min_{s\in[s_0,\bar\beta]}Q(s)\Bigr)\Bigr)+\frac{(N-2)^2\bar\beta^2}{2\bar C^2}$$
if $f$ satisfies $(A2)$ where
$$\bar C:=2(N-1)\frac{\bar\beta-\beta_1}{F(\bar\beta)-F(\gamma_M)}(2(F(\bar\beta)-\min_{s\in[\beta_1,\beta_{*}]}F(s)))^{1/2},$$
 We have
\begin{theorem}\label{main2}  If  $f$ satisfies assumptions $(A1)$ or $(A2)$ and
\beq\label{caso1}
(\bar C+A)\bar I
<\frac{2^{1/2}(N-1)}{(\bar I+\bar F)^{1/2}}\int_0^{\beta_1} |F(s)|ds,
\eeq
then for any $k\in\mathbb N\cup\{0\}$ there exist at least two solutions $u$ of \eqref{eq2}, with initial value in $(\beta_*,\gamma_*)$, having exactly $k$ sign changes in $(0,\infty)$.
\end{theorem}
 \begin{remark}\rm
 If $f$ satisfies $(A2)$ and
 \ben
F_\infty:= \lim\limits_{s\to\infty}F(s)<\infty,
\een
then the above theorem holds with $\bar I=F_\infty$.
 \end{remark}
We will obtain our results through a careful study of the  initial value problem
\begin{equation}\label{ivp}
\begin{gathered}
u''+\frac{N-1}{r}u'+f(u)=0,\quad r>0,\quad N\ge2,\\
u(0)=\alpha,\quad u'(0)=0,
\end{gathered}\end{equation}
for $\alpha\in(\beta_*,\gamma_*)$. By a solution to \eqref{ivp} we mean a $C^1$ function $u$ such that $u'$ is also $C^1$ in its domain and we denote such a solution by $u(\cdot,\alpha)$.

 The idea of the proof of our multiplicity result is to define the set
$\mathcal Q_1$ as the set of initial values $\alpha>\beta_*$ such that the corresponding solution $u(\cdot,\alpha)$ of \eqref{ivp} is strictly positive and $\inf_{r>0} u(r,\alpha)\in(0,\beta_1)$.
We extend this definition to the similar  sets $\mathcal Q_k$ when the solution $u(\cdot,\alpha)$ of \eqref{ivp} has exactly $k-1$ zeros. By continuous dependence of solutions in the initial data, it will follow that $\mathcal Q_k$ is an open set. Let $\mathcal G_k$  be the set of initial values $\alpha>\beta_*$ such that the corresponding solution $u(\cdot,\alpha)$ is a solution of \eqref{eq2} having exactly $k-1$ simple zeros in $(0,\infty)$.

In some of previous works concerning existence of solutions, see for example \cite{fls, gst} for ground states and \cite{cghy4, cdghm} for higher order bound states having a prescribed number of nodes, the conditions on $f$ imply that $F$ does not possess a positive local maximum, hence $\mathcal Q_k$ in nonempty for all $k$ and $\sup(\mathcal Q_k\cup\mathcal G_k)\in\mathcal G_k$. On the other hand, $\inf(\mathcal Q_k\cup\mathcal G_k)$ in general does not belong to $\mathcal G_k$, in fact there are cases for which there is uniqueness, that is, $\mathcal G_k$ is a singleton.

The presence of a positive local maximum for $F$ ($(f_3)$ in our assumptions) will guarantee that if $\mathcal Q_k$ is nonempty,
then $\inf(\mathcal Q_k\cup\mathcal G_k)$ and $\sup(\mathcal Q_k\cup\mathcal G_k)$ are different and belong to $\mathcal G_k$.  Theorem \ref{main0} will follow once we have proved that $\mathcal Q_k$ is nonempty for $k$ large enough. A striking difference with the case for which $F$ does not possess a positive local maximum is that now $\mathcal Q_1$ can be empty. This result is contained in Theorem \ref{main1}. Finally, in Theorem \ref{main2} we give conditions on $f$ so that $\mathcal Q_1\not=\emptyset$.
\medskip

This paper is organized as follows. In section \ref{prel}, we establish some properties of the solutions to \eqref{ivp}, we restrict its domain  to the set of unique extendibility,   define some crucial sets of initial values and prove some crucial results concerning the solutions having initial value in these sets.  Then in section \ref{exist} we prove our main result. Finally in the Appendix we prove a non-oscillation result for the solutions of \eqref{ivp}.

\section{Some properties of the solutions of the initial value problem}\label{prel}

The aim of this section is to establish several properties of the solutions to the initial value problem \eqref{ivp}.
Since $f$ is continuous, problem \eqref{ivp} has a solution defined for all $r\ge 0$ for any $\alpha>\beta_*$  but it might not be uniquely defined. It is straight forward to see that unique extendibility can be lost only if $u$ reaches a double zero.

\begin{defi}\label{domain}
The domain $D$ of definition of $u$ will be the domain of unique extendibility.
\end{defi}

 That is, $D=(0,D_\alpha)$,
where if
$
D_\alpha<\infty$,
then $D_\alpha$ is a double zero of $u$.

 By standard theory of ordinary differential equations, the solution depends continuously on the initial data in any compact subset of its domain of definition.

We start by stating without proof the following basic proposition. The proof of $(i)$ and $(iii)$ can be found in \cite[Proposition 2.3]{cghy4} and the proof of $(ii)$ can be found in \cite[Proposition 3.4]{dghm}. A proof of $(iv)$ under other assumptions can be found in \cite{cdghm}, we include a proof of it under the new assumptions in the Appendix. These proofs are based on properties of the well known energy functional
\begin{eqnarray*}
I(r,\alpha)=\frac{|u'(r,\alpha)|^2}{2}+F(u(r,\alpha))\end{eqnarray*}
for which we have
\beq\label{Ider}
I'(r,\alpha)=-(N-1)\frac{|u'(r,\alpha)|^2}{r}.
\eeq
\begin{proposition}\label{basic} Let $f$ satisfy $(f_1)$-$(f_2)$ in either $(A1)$ or $(A2)$  and let $u(\cdot,\alpha)$ be a solution of \eqref{ivp}.
\begin{enumerate}
\item[(i)] There exists $C(\alpha)>0$ such that $|u(r,\alpha)|+|u'(r,\alpha)|\le C(\alpha)$.
\item[(ii)] $\lim_{r\to\infty}I(r,\alpha)$ exists and is equal to $F(\ell)$, where $\ell$ is a zero of $f$.
\item[(iii)]
If  $u(\cdot,\alpha)$ is   defined in $[0,\infty)$ and $\lim_{r\to \infty}u(r,\alpha)=\ell$, then
$$\quad \lim_{r\to \infty}u'(r,\alpha)=0\quad \mbox{and}\quad \ell\ \mbox{is a zero of $f$}.$$
\item[(iv)] Assume further that $f$ satisfies $(f_4)$ of either $(A1)$ or $(A2)$. Then $u$ has at most a finite number of sign changes.
\end{enumerate}
\end{proposition}

Let us set
$$Z_1(\alpha):=\sup\{r>0\ |\ u(s,\alpha)>0\mbox{ and }u'(s,\alpha)<0\ \mbox{ for all }s\in(0,r)\}$$
and define
\begin{eqnarray*}
{\mathcal N_1}&=&\{\alpha\in[\beta_*,\gamma_*)\ :\ u(Z_1(\alpha),\alpha)=0\quad\mbox{and}\quad u'(Z_1(\alpha),\alpha)<0\}\\
{\mathcal G_1}&=&\{\alpha\in[\beta_*,\gamma_*)\ :\ u(Z_1(\alpha),\alpha)=0\quad\mbox{and}\quad u'(Z_1(\alpha),\alpha)=0\}\\
{\mathcal P_1}&=&\{\alpha\in[\beta_*,\gamma_*)\ :\ u(Z_1(\alpha),\alpha)>0\},
\end{eqnarray*}
where $\beta_*$ is as defined in Definition \ref{constants}(ii), and we recall $\gamma_*=\infty$ in case $f$ satisfies $(A2)$. We now extend these definitions by induction for $k\ge2$.

If ${\mathcal N_{k-1}}\not=\emptyset$, we set
$${\mathcal F}_k=\{\alpha\in\mathcal N_{k-1}\ :\ (-1)^ku'(r,\alpha)\le 0\quad\mbox{for all }r> Z_{k-1}(\alpha)\}.$$
For $\alpha\in \mathcal N_{k-1}\setminus{\mathcal F}_k$, we set
\ben
T_{k-1}(\alpha):&=&\sup\{r\in(Z_{k-1}(\alpha),D_\alpha)\ :\ (-1)^ku'(r,\alpha)\le 0\},
\een
and for $\alpha\in{\mathcal F}_k$, we set $T_{k-1}(\alpha)=\infty$.

Next, for $\alpha\in \mathcal N_{k-1}\setminus {\mathcal F}_k$, we define the extended real number
\begin{eqnarray*}
Z_k(\alpha):=\sup\{r>T_{k-1}(\alpha)\ |\ (-1)^ku(s,\alpha)<0\mbox{ and }(-1)^ku'(s,\alpha)>0\ \\
\mbox{ for all }s\in(T_{k-1}(\alpha),r)\},
\end{eqnarray*}
and again if $\alpha\in{\mathcal F}_k$, we set $Z_k(\alpha)=\infty$.
\begin{figure}[h]
\begin{center}
 \includegraphics[keepaspectratio, width=13cm]{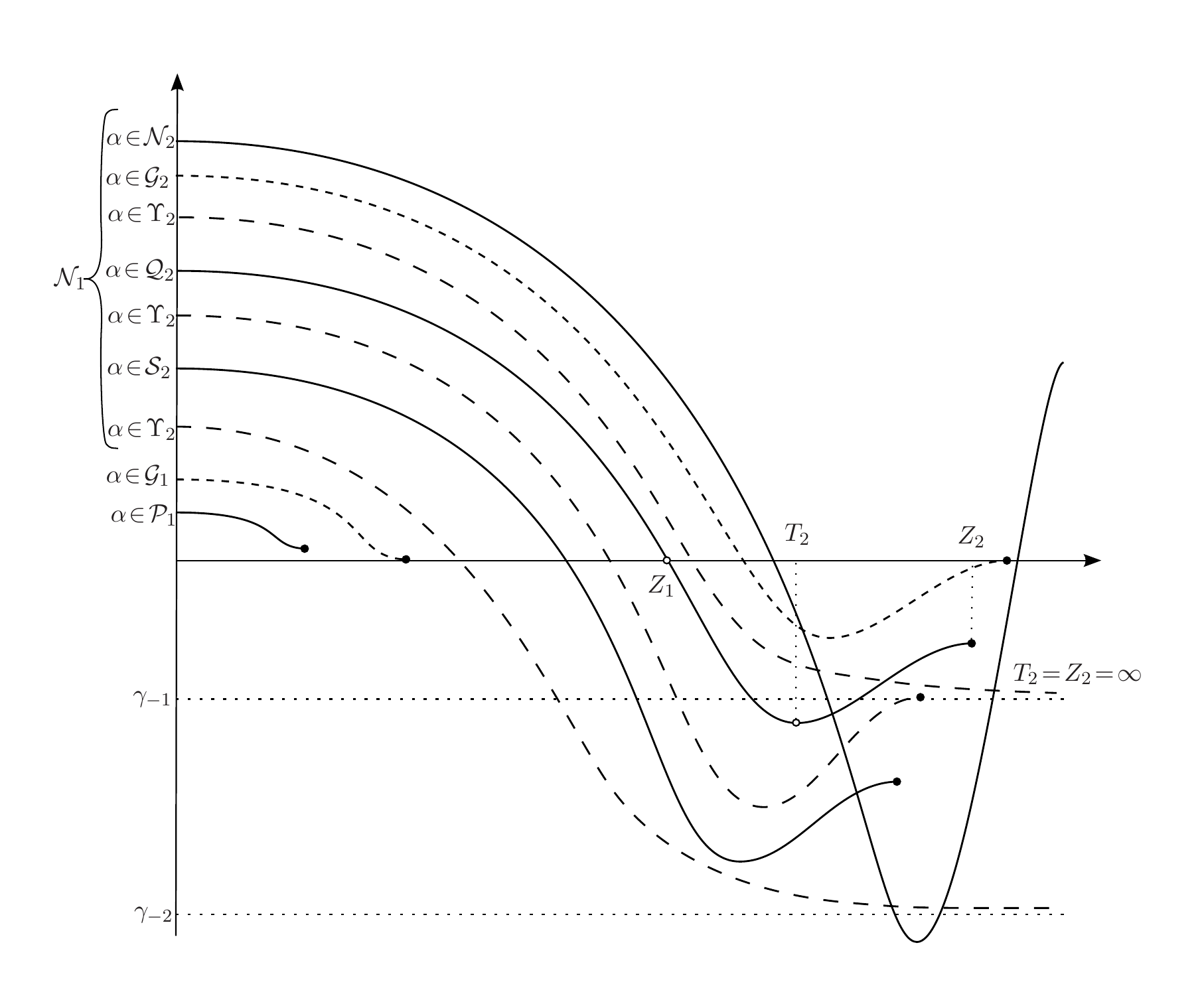}
 \end{center}
 \caption{Solutions of \eqref{ivp} with initial condition in these sets}
\end{figure}

We now define
\begin{eqnarray*}
{\mathcal N_k}&=&\{\alpha\in\mathcal N_{k-1}\setminus{\mathcal F}_k\ :\ u(Z_k(\alpha),\alpha)=0\quad\mbox{and}\quad (-1)^ku'(Z_k(\alpha),\alpha)>0\},\\
{\mathcal G_k}&=&\{\alpha\in\mathcal N_{k-1}\setminus {\mathcal F}_k\ :\ u(Z_k(\alpha),\alpha)=0\quad\mbox{and}\quad u'(Z_k(\alpha),\alpha)=0\},\\
{\mathcal P_k}&=&\{\alpha\in\mathcal N_{k-1}\ :\ (-1)^ku(Z_k(\alpha),\alpha)<0\}.
\end{eqnarray*}
Finally, for any $k\in\mathbb N$ we decompose $\mathcal P_k$ as follows:
$$\mathcal P_k=\mathcal Q_k\cup \mathcal S_k\cup \Upsilon_k,$$
where
\begin{eqnarray*}
\mathcal Q_k&=&\{\alpha\in\mathcal P_k\ :\ \gamma_{-1}<u(Z_k(\alpha),\alpha)<0 \quad\mbox{or}\quad 0<u(Z_k(\alpha),\alpha)<\gamma_1\}\\
\mathcal S_k&=&\bigcup_{i=\bar M-1,\ i\neq 0,-1}^{M}\{\alpha\in\mathcal P_k\ :\ \gamma_i<u(Z_k(\alpha),\alpha)<\gamma_{i+1}\}\\
\Upsilon_k&=&\bigcup_{i=\bar M,\ i\neq 0}^{M}\{\alpha\in\mathcal P_k\ :\ u(Z_k(\alpha),\alpha)=\gamma_i\}
\end{eqnarray*}
where the constants $\gamma_i$ are defined in Definition \ref{constants}(i).

It should be noticed that  if $\alpha\in\Upsilon_k$, then necessarily $Z_k(\alpha)=\infty$. Indeed, let $\alpha\in\Upsilon_k$ and assume $Z_k(\alpha)<\infty$. Then $u'(Z_k(\alpha),\alpha)=0$ and $u(Z_k(\alpha),\alpha)=\gamma_i$ for some $i\not=0$. By the unique solvability of \eqref{ivp} up to a double zero, it must be that  $u(r)\equiv \gamma_i$ for all $r\ge Z_k(\alpha)$. But then we can argue as in the proof of \cite[Proposition 1.3.1]{fls}
to obtain a contradiction to the fact that by the Lipschitz assumption on $f$, we have that
$$\int_{\gamma_i}\frac{du}{|F(\gamma_i)-F(u)|^{1/2}}=\infty.$$
\medskip

As the minima (maxima) of $u$ occur at values where $f(u)\le 0$ ($f(u)\ge 0$), it follows that if $\alpha\in \mathcal S_j\cup\mathcal Q_j$, with $\gamma_i<u(Z_j(\alpha),\alpha)<\gamma_{i+1}$, then $F(u(Z_j(\alpha),\alpha))<\min\{F(\gamma_i),F(\gamma_{i+1})\}$ and hence $\gamma_i<u(r,\alpha)<\gamma_{i+1}$ for all $r>Z_j(\alpha)$.

The rest of this section is devoted to the proof of some crucial properties of the sets defined above.
\medskip

\begin{lemma}\label{p2m}
Assume that $f$ satisfies $(A1)$ or $(A2)$ and let $k\in\mathbb N$.
\begin{enumerate}
\item[(i)]If $\bar\alpha\in\mathcal G_k$, then there exists a neighborhood $V$ of $\bar\alpha$ such that
if $\alpha\in V\cap\mathcal N_k$, then $\alpha\in \mathcal Q_{k+1}$
\item[(ii)]If $\bar\alpha$ is such that $u(Z_k(\bar\alpha))=\gamma$, with $\gamma$ a   local maximum of $F$ with $F(\gamma)\ge0$, then there exists a neighborhood $V$ of $\bar\alpha$ such that
if $\alpha\in V\cap\mathcal N_k$, then  $F(u(T_k(\alpha),\alpha))<F(\gamma)$.
\item[(iii)]Let $\bar\alpha$ be such that $u(Z_k(\bar\alpha))=\gamma$, with $\gamma$ a  local maximum of $F$ with $F(\gamma)\ge0$, and $\gamma_i<\gamma<\gamma_{i+1}$.
 Then there exists a neighborhood $V$ of $\bar\alpha$, such that if $\alpha\in V$, then $\alpha\in\mathcal N_{k-1}\setminus \mathcal N_k$ and either $u(Z_k(\alpha))=\gamma$ or there exists $r_1>0$ such that $\gamma_i<u(r_1,\alpha)<\gamma_{i+1}$ and $I(r_1,\alpha)<\min\{F(\gamma_i),F(\gamma_{i+1})\}$.

\end{enumerate}

\end{lemma}

\begin{proof}
\noindent Part $(i)$: Let $\bar\alpha\in\mathcal G_k$.  Without loss of generality we may assume that $u(\cdot,\bar\alpha)$ is decreasing in $(T_{k-1}(\bar\alpha),Z_k(\bar\alpha))$. We will show that there exists a neighborhood $V$ such that if $\alpha\in V\cap\mathcal N_k$, then $u(T_k(\alpha),\alpha)>\beta_{-1}$. Arguing by contradiction we assume  that there exists a sequence
$\{\alpha_i\}$,
 $\alpha_i\to\bar\alpha$  with $\alpha_i\in\mathcal N_k$,  such that
\beq\label{n1} u(T_k(\alpha_i),\alpha_i)\le\beta_{-1}
\eeq
so that $u(\cdot,\alpha_i)$ has crossed the value $-\delta$ with positive energy.

Let now $\varepsilon\in(0,1)$. Since
$$\lim_{r\to Z_k(\bar\alpha)}I(r,\bar\alpha)=0\quad\mbox{and}\quad \lim_{r\to Z_k(\bar\alpha)}u(r,\bar\alpha)=0,$$
there exists $r_0>T_{k-1}(\bar\alpha)$ such that
$$I(r_0,\bar\alpha)<\varepsilon,\quad 0<u(r_0,\bar\alpha)<\delta/2,$$
where $\delta$ is as defined in $(f_2)$ of $(A1)$ and $(A2)$, and therefore by continuity, for  $i$ large enough, $0<u(r_0,\alpha_i)<\delta$, $Z_k(\alpha_i)>r_0$, and
$$I(r_0,\alpha_i)<2\varepsilon. $$
Since $I$ is decreasing in $r$, we have that
$$I(r,\alpha_i)<2\varepsilon\quad\mbox{for all}\quad r\in (r_0,T_{k}(\alpha_i)),$$
hence,
\beq\label{n2}
|u'(r,\alpha_i)|\le \sqrt{4-2\min_{s\in[\beta_{*}^-,\beta_*]}F(s)}:=K\quad\mbox{for all }r\in (r_0,T_{k}(\alpha_i))
\eeq
and $i$ large enough. Let us denote by $r(\cdot,\alpha_i)$ the inverse of $u(\cdot,\alpha_i)$ in
$(T_{k-1}(\alpha_i),T_k(\alpha_i))$. From \eqref{n1}, $[-\delta,0]\subset[u(T_k(\alpha_i),\alpha_i),0]$, and from \eqref{n2}, by the mean value theorem we obtain that
\ben\Bigl(\frac{-\delta}{2},\alpha_i\Bigr)-r\Bigl(\frac{-\delta}{4},\alpha_i\Bigr)\ge \frac{\delta}{4K}.
\een
 Let now
$$H(r,\alpha)=r^{2(N-1)}I(r,\alpha).$$
Then
$$H'(r,\alpha)=2(N-1)r^{2N-3}F(u(r,\alpha)),$$
implying that for $\alpha=\bar\alpha$, $H'(r,\bar\alpha)< 0$ for all $r\in(r_0, Z_k(\bar\alpha))$ and
$$H(r,\bar\alpha)\downarrow L\ge 0$$
as $r\to Z_k(\bar\alpha)$. Also, by choosing a larger $r_0$ if necessary, we may assume $H(r_0,\bar\alpha)<L+\varepsilon$. Thus by continuity we have that
$$H(r_0,\alpha_i)\le L+2\varepsilon\quad\mbox{for $i$ large enough.}$$
Also, as $u(r,\alpha_i)<\delta$ for $r\in[r_0,Z_k(\alpha_i)]$, $H$ is decreasing in $[r_0,Z_k(\alpha_i)]$ implying
$$H(Z_k(\alpha_i),\alpha_i)\le L+2\varepsilon\quad\mbox{for $i$ large enough.}$$
Integrating $H'(\cdot,\alpha_i)$ over $(Z_k(\alpha_i),r(\frac{-\delta}{2},\alpha_i))$, we find that
$$
H(r\Bigl(\frac{-\delta}{2},\alpha_i\Bigr),\alpha_i)-H(Z_k(\alpha_i),\alpha_i)=-2(N-1)
\int_{Z_k(\alpha_i)}^{r(\frac{-\delta}{2},\alpha_i)}t^{2N-3}|F(u(t,\alpha_i))|dt$$
and thus, observing that since $N\ge 2$, we have $2N-3>0$ implying
\ben
H(r\Bigl(\frac{-\delta}{2},\alpha_i\Bigr),\alpha_i)&\le& L+2\varepsilon-2(N-1)(Z_k(\alpha_i))^{2N-3}\int_{Z_k(\alpha_i)}^{r(\frac{-\delta}{2},\alpha_i)}|F(u(t,\alpha_i))|dt\\
&\le& L+2\varepsilon-2(N-1)(Z_k(\alpha_i))^{2N-3}\int_{r(\frac{-\delta}{4},\alpha_i)}^{r(\frac{-\delta}{2},\alpha_i)}|F(u(t,\alpha_i))|dt\\
&\le& L+2\varepsilon-2(N-1)(Z_k(\alpha_i))^{2N-3}(r\Bigl(\frac{-\delta}{2},\alpha_i\Bigr)-r\Bigl(\frac{-\delta}{4},\alpha_i\Bigr))C\\
&\le& L+2\varepsilon-2(N-1)(Z_k(\alpha_i))^{2N-3}\frac{\delta}{4K}C,
\een
where $C:=\inf\{|F(s)|,\ s\in[\frac{-\delta}{2},\frac{-\delta}{4}]\}$. If $Z_k(\bar\alpha)=\infty$, by taking $i$ larger if necessary, we conclude that $H(r(\frac{-\delta}{2},\alpha_i),\alpha_i)<0$, a contradiction.
If $Z_k(\bar\alpha)<\infty$, the same conclusion follows by observing that in this case $L=0$ and $Z_k(\alpha_i)$ is bounded below by the positive constant $ r_1/2$, where $ r_1$ the first value of $r>0$ where $u(\cdot,\bar\alpha)$ takes the value $\delta$.
\medskip

\noindent Part $(ii)$: The proof is very similar to that of Part $(i)$, the only difference is that now we consider
\beq\label{tildeH}\tilde H(r,\alpha)=r^{2(N-1)}(I(r,\alpha)-F(\gamma)),
\eeq
so that
\ben\tilde H'(r,\alpha)=2(N-1)r^{2N-3}(F(u(r,\alpha))-F(\gamma)).
\een
We still assume that $u(\cdot,\bar\alpha)$ is decreasing in $(T_{k-1}(\bar\alpha),Z_k(\bar\alpha))$ and that
$\{\alpha_i\}$ contains a subsequence, still denoted the same, such that
\ben F(u(T_k(\alpha_i),\alpha_i))\ge F(\gamma)
\een
 so that $u(\cdot,\alpha_i)$ has crossed the value $-\delta$ with   energy greater than $F(\gamma)$. As above,
$[-\delta,0]\subset[u(T_k(\alpha_i),\alpha_i),0]$, and from  the mean value theorem we obtain that
\ben r\Bigl(\frac{-\delta}{2},\alpha_i\Bigr)-r\Bigl(\frac{-\delta}{4},\alpha_i\Bigr)\ge \frac{\delta}{4K},
\een
where now  $K:= \sqrt{2(F(\gamma)+2-2\min_{s\in[\beta_{*}^-,\beta_*]}F(s))}$. Setting $C_0:=\inf\{|F(s)-F(\gamma)|,\ s\in[\frac{-\delta}{2},\frac{-\delta}{4}]\}$ and
$0\le L:=\lim\limits_{r\to Z_k(\bar\alpha)}\tilde H(r,\bar\alpha)$,
we obtain
$$\tilde H(r\Bigl(\frac{-\delta}{2},\alpha_i\Bigr),\alpha_i)\le L+2\varepsilon-2(N-1)(Z_k(\alpha_i))^{2N-3}\frac{\delta}{4K}C_0.$$
The same reasoning as above leads to the conclusion that for $i$ sufficiently large
$$I(r\Bigl(\frac{-\delta}{2},\alpha_i\Bigr),\alpha_i)<F(\gamma),$$
a contradiction to the fact that $I$ is decreasing.

\noindent{ Part $(iii)$:} If $\gamma_i<\gamma<\gamma_{i+1}$,  and since $F(\gamma)\le \min\{F(\gamma_i),F(\gamma_{i+1})\}$, we can repeat the same argument as above but replacing the interval $[-\delta/2,-\delta/4]$ by an interval $[a,b]\subset(\gamma_i,\gamma)$ if $i\ge0$ and $[a,b]\subset(\gamma,\gamma_{i+1})$ if $i\le-1$, where $F(s)<F(\gamma)$.
\end{proof}

Our next result is a generalization of Lemma 3.1 in \cite{gst}.

\begin{lemma}\label{gamma-vec}
\mbox{ }

\begin{enumerate}
\item[(i)]
Let $f$ satisfy $(A1)$ or $(A2)$, and let $\bar\alpha$ such that $u(Z_j(\bar\alpha),\bar\alpha)=\gamma_i$ for some $i\neq 0, \bar M -1$, and let $k\ge j$. Then there exists a neighborhood $V_k$ of $\bar\alpha$ such that if $\alpha\in V_k$ and $u(Z_j(\alpha),\alpha)\not=\gamma_i$, then $\alpha\in \mathcal N_k$.
\item[(ii)]
Let $f$ satisfy $(A2)$,   $\bar\beta$ be defined as in Definition \ref{constants}(iii),  $\bar u\in(\gamma_*^-,0)$ such that $F(\bar u)=F(2\bar\beta)$ and set $-\tilde F:=\min_{s\in[\beta_*^-,\beta_*]}F(s)$.  If $\alpha>2\bar\beta$, with $\alpha\in \mathcal N_j$ for $j\le k$ and
$$\bar r(\alpha)\ge C_k:=\frac{(k+1)(2\bar\beta-\bar u)(N-1)\sqrt{2(F(2\bar\beta)+\tilde F)}}{F(2\bar\beta)-F(\bar\beta)},$$
where $\bar r(\alpha)$ denotes the first point after $T_{j-1}(\alpha)$ for which $F(u(\bar r(\alpha),\alpha))=F(2\bar\beta)$,  then $\alpha\in\mathcal N_k$.
\end{enumerate}
\end{lemma}
\begin{proof}
\mbox{ }
\noindent Part (i): Without loss of generality we may assume that $u(Z_j(\bar\alpha),\bar\alpha)=\gamma_i>0$.
Let
$$B_i=\max\{F(\gamma_\ell)\ |\ F(\gamma_\ell)<F(\gamma_i)\}$$
and $u_i$ be the largest point in $(\gamma_*^-,0)$ such that $F(\gamma_i)=F(u_i)$.
Set
$$\varepsilon:=\frac{F(\gamma_i)-B_i}{k+1}.$$
Let $D_1,\ D_2$ be such that
$$D_1:=\frac{(\gamma_i-u_i)(N-1)\sqrt{2(F(\gamma_i)+\tilde F)}}{\varepsilon},\quad F(u(D_2,\bar\alpha))>F(\gamma_i)-\frac{\varepsilon}{2},$$
and set $D:=\max\{D_1,D_2\}$. By  the continuous dependence of the solutions on the initial data and Lemma \ref{p2m}(ii), there exists a neighborhood $V$ of $\bar\alpha$ such that for $\alpha\in V$,
$$\sup_{r\in[0,D]}|F(u(r,\alpha))-F(u(r,\bar\alpha))|<\varepsilon/2,$$
and if $\alpha\in \mathcal N_j$, $F(u(T_j(\alpha),\alpha))<F(\gamma_i)$.  Let now $\alpha\in V$ and assume that $u(Z_j(\alpha),\alpha)\not=\gamma_i$, and denote by $\bar r_{\varepsilon}$ the first point after $D$ such that $F(u(\bar r_{\varepsilon},\alpha))=F(\gamma_i)-\varepsilon$. Denote by $r_0:=r_0(\alpha)$  the first point after $\bar r_{\varepsilon}$ where $u'(r_0,\alpha)=0$.
By integrating \eqref{Ider} over $(\bar r_{\varepsilon}, r_0)$ we find that
$$I(\bar r_{\varepsilon},\alpha)-F(u(r_0,\alpha))=(N-1)\int_{\bar r_{\varepsilon}}^{r_0}\frac{|u'(r,\alpha)|^2}{r}dr,$$
hence, using that
$$|u'(r,\alpha)|\le\sqrt{2(I(\bar r_{\varepsilon})+\tilde F)}\quad\mbox{for all }r>\bar r_{\varepsilon}$$
we obtain
$$F(u(r_0,\alpha))\ge I(\bar r_{\varepsilon})\Bigl(1-\frac{\sqrt{2(I(\bar r_{\varepsilon})+\tilde F)}}{I(\bar r_{\varepsilon})}\frac{(\gamma_i-u_i)(N-1)}{\bar r_{\varepsilon}}\Bigr).$$
Therefore, as
$\sqrt{2(I+\tilde F)}/I$ is decreasing in $I$, $I(\bar r_{\varepsilon})\ge F(\gamma_i)-\varepsilon$,  $\bar r_{\varepsilon}>D$ and $\varepsilon<F(\gamma_i)/(k+1)$, we have that
\ben
F(u(r_0,\alpha))&\ge& I(\bar r_{\varepsilon})\Bigl(1-\frac{\sqrt{2(F(\gamma_i)-\varepsilon+\tilde F)}}{F(\gamma_i)-\varepsilon}\frac{\varepsilon}{\sqrt{2(F(\gamma_i)+\tilde F)}}\Bigr)\\
&\ge & F(\gamma_i)-2\varepsilon.
\een
Hence,
$$F(\beta_i)\le B_i<F(\gamma_i)-2\varepsilon<F(u(r_0,\alpha))<F(\gamma_i).$$
Since $f(s)>0$ for $s\in(\beta_i,\gamma_i)$, we deduce that  $r_0=T_j(\alpha)$.
Iterating this process at $\bar r_{2\varepsilon}$, the first point after $T_j(\alpha)$ at which  $F(u(\bar r_{2\varepsilon},\alpha))=F(\gamma_i)-2\varepsilon$,  we obtain $\alpha\in\mathcal N_2$. We repeat this procedure $k$ times to obtain $\alpha\in\mathcal N_k$.

\noindent Part (ii): Without loss of generality we may assume that $u(\bar r(\alpha),\alpha)>0$. Let
$$\varepsilon:=\frac{F(2\bar\beta)-F(\bar\beta)}{k+1},$$
and again denote by  $r_0:=r_0(\alpha)$  the first point after $\bar r(\alpha)$ where $u'(r_0,\alpha)=0$.
By integrating \eqref{Ider} over $(\bar r(\alpha), r_0)$ as in Part $(i)$ we obtain
$$F(\bar\beta)<F(2\bar\beta)-\varepsilon<F(u(r_0,\alpha)),$$
and therefore $r_0=T_j(\alpha)$.
Iterating this process at $\bar r_{\varepsilon}$, the first point after $T_j(\alpha)$ at which  $F(u(\bar r_{\varepsilon},\alpha))=F(\gamma_i)-\varepsilon$,  we obtain $\alpha\in\mathcal N_2$. We repeat this procedure $k$ times to obtain $\alpha\in\mathcal N_k$.
\end{proof}

\begin{lemma}
\label{front}
\mbox{ }
\begin{enumerate}
\item[(i)]
The sets ${\mathcal N_k}$, ${\mathcal Q_k}$ and ${\mathcal S_k}$ are open in $[\beta_*,\gamma_*)$.
\item[(ii)]
The boundary of $\mathcal G_k\cup\mathcal Q_k$ is contained in $\bigcup_{i=1}^k\mathcal G_i$.
\end{enumerate}
\end{lemma}
\begin{proof}
\mbox{ }\\
\noindent{Part (i):} The proof that $\mathcal N_k$ is open follows by continuous dependence of solutions in the initial value $\alpha$, see \cite[Proposition 2.4]{cghy4}.

 Let now $k\ge 1$ and let $\bar\alpha\in\mathcal Q_k$. Without loss of generality we may assume  $0<u(Z_k(\bar\alpha),\bar\alpha)<\gamma_1$. If $I(Z_k(\bar\alpha),\bar\alpha)<0$, then there exists $r_1>0$ such that $I(r_1,\bar\alpha)<0$ and $0<u(r_1,\bar\alpha)<\gamma_1$. By continuous dependence of solutions in the initial data, there exists $\delta>0$ such that for any $\alpha\in(\bar\alpha-\delta,\bar\alpha+\delta)$, then $I(r_1,\alpha)<0$ and $0<u(r_1,\alpha)<\gamma_1$. Moreover, by taking a smaller $\delta$ if necessary, we have that $u(\cdot,\alpha)$ has exactly $k-1$ zeros in $[0,r_1]$, hence $(\bar\alpha-\delta,\bar\alpha+\delta)\subset\mathcal Q_k$.\\
 If $I(Z_k(\bar\alpha),\bar\alpha)=0$, then $u(Z_k(\bar\alpha),\bar\alpha)$ is a local maximum of $F$ and the result follows from Lemma \ref{p2m} (iii).\\
 The same argument shows that $\mathcal S_k$ is open.

\noindent{Part (ii):}
As $\mathcal N_k$ is open, we have that $\mathcal N_k\cap\overline{\mathcal Q_k\cup\mathcal G_k}=\emptyset$.

Let $\bar\alpha$ belong to the boundary of $\mathcal Q_k\cup\mathcal G_k$. As $\mathcal Q_i$ and $\mathcal S_i$ are open, we must have that    $\bar\alpha\in\bigcup_{i=1}^k\mathcal G_i\cup\Upsilon_i$. But from Lemma \ref{gamma-vec}, if $\bar\alpha\in\bigcup_{i=1}^k\Upsilon_i$, then there exists $\delta>0$ such that $V_{\delta}(\bar\alpha)\subset \Upsilon_j\cup\mathcal N_k$, implying that
$(\mathcal Q_k\cup\mathcal G_k)\cap V_{\delta}(\bar\alpha)=\emptyset$, a contradiction. Hence
$\bar\alpha\in \bigcup_{i=1}^k\mathcal G_i$.
\end{proof}

\section{Proof of the main results }\label{exist}

In this section we prove our theorems. To this end, we need the following key result, which is a generalization of Lemma 3.1 in \cite{cghy4}.

\begin{lemma}\label{nn2}
Assume that $f$ satisfies $(A1)$ or $(A2)$. Then,  for each $k\in\mathbb N$, there exists  $\alpha_k\in(\beta_*,\gamma_*)$ such that $[\alpha_k,\gamma_*)\subset{\mathcal N}_k$.

\end{lemma}

\begin{proof}
Assume first that $f$ satisfies $(A1)$. We apply Lemma \ref{gamma-vec} to $\bar\alpha=\gamma_*$, $\gamma_i=\gamma_*$ and $j=1$ to obtain that there exists $\alpha_k>0$ such that $[\alpha_k,\gamma_*)\subset\mathcal N_k$.

 \noindent Let $f$ satisfy $(A2)$.
We will use here a useful and well known Pohozaev type identity
which plays a key role in this proof.
For a solution $u(\cdot,\alpha)$ of \eqref{ivp}, set
$$E(r,\alpha):=2r^NI(r,\alpha)+(N-2)r^{N-1}u'(r,\alpha)u(r,\alpha).$$
Then
\beq\label{energy}E'(r,\alpha)=r^{N-1}Q(u(r,\alpha)).
\eeq
 Let $k\in\mathbb N$, let $\bar\beta$ be as defined in Definition \ref{constants}(iii). By Lemma \ref{gamma-vec}(ii), if for $\alpha>2\bar\beta$ it holds that $\bar r:=\bar r(\alpha)\ge C_k$, then $\alpha\in\mathcal N_k$.

 Assume that $\alpha\ge2\bar\beta$ and  $\bar r(\alpha)< C_k$.
 Let $\theta\in(0,1)$ be as in assumption $(f_5)$ and let $\alpha$ be large enough to have $\theta\alpha>2\bar\beta$. By setting  $r_\theta >0$  the first point where $u(r_\theta,\alpha)=\theta \alpha$, integration of \eqref{energy} over $[0,\bar r]$ yields
\ben
E(\bar r,\alpha)&\ge& (\int_{0}^{r_\theta}+\int_{r_\theta}^{\bar r})t^{N-1}Q(u(t,\alpha))dt\nonumber\\
 &\ge& \int_{0}^{r_\theta}t^{N-1}Q(u(t,\alpha))dt\quad\mbox{ (as $Q(u(t,\alpha))\ge 0 $ in $[r_\theta,\bar r]$) }\\
&\ge& Q(s_2)\frac{r_\theta^N}{N}\quad\mbox{ where we have set $Q(s_2)=\min_{s\in[\theta\alpha,\alpha]}Q(s)$.}\nonumber
\een
Now we estimate $r_\theta$:
Set $f(s_1)=\max_{s\in[\theta \alpha,\alpha]}f(s)$ ($s_1\in [\theta \alpha,\alpha]$). From the equation in \eqref{ivp}, we obtain, as in \cite{cghy4}
$$r_\theta\ge \Bigl(\frac{c\alpha}{f(s_1)}\Bigr)^{1/2},$$
where $c=2N(1-\theta)$. Therefore, by $(f_5)$ we conclude that
\ben
E(\bar r,\alpha)\ge \frac{1}{N}Q(s_2)\Bigl(\frac{c\alpha }{f(s_1)}\Bigr)^{N/2}\to\infty\mbox{ as $\alpha\to\infty$.}
\een
Let us choose $\alpha_k$ such that for $\alpha>\alpha_k$,
$$E(\bar r,\alpha)\ge 2(C_k+1)^NB+(k+1)\bar Q\frac{(C_k+1)^N}{N}$$
where $\bar Q:=-\min_{s\in[s_0,\bar\beta]}Q(s)\ge0$,  let $\bar u$ be the unique point in $(\gamma_*^-,0)$ such that $F(\bar u)=F(2\bar\beta)$ and set
\ben B=\Bigl(4\bar\beta-2\bar u+\frac{(N-2)|\bar u|}{2(C_k+1)}\Bigr)^2+F(2\bar\beta).
\een
 Let now $\alpha\ge\alpha_k$ and let $r_0=r_0(\alpha)$ be the first point after $\bar r(\alpha)$ such that either
$$r_0=C_k+1,\quad\mbox{or}\quad u'(r_0,\alpha)=0,\quad\mbox{or}\quad F(u(r_0,\alpha))=F(2\bar\beta).$$
As $r_0\le C_k+1$, for $r\le r_0$ we have
\ben
E(r,\alpha)&=&E(\bar r,\alpha)+\int_{2\bar\beta}^{r}t^{N-1}Q(u(t,\alpha))dt\\
&\ge& E(\bar r,\alpha)-\bar Q\frac{(C_k+1)^N}{N}.
\een
implying
\beq\label{p3}
E(r,\alpha)\ge 2(C_k+1)^NB+k\bar Q\frac{(C_k+1)^N}{N}
\eeq and thus
\ben
2I(r,\alpha)+\frac{(N-2)|\bar u|}{C_k+1}|u'(r,\alpha)|\ge 2B.
\een
We deduce that
\ben
\Bigl(|u'(r,\alpha)|+\frac{(N-2)|\bar u|}{2(C_k+1)}\Bigr)^2&\ge& |u'(r,\alpha)|^2+\frac{(N-2)|\bar u||u'(r,\alpha)|}{C_k+1}\\
&\ge& 2B-2F(\bar u)
=\Bigl(4\bar\beta-2\bar u+\frac{(N-2)|\bar u|}{2(C_k+1)}\Bigr)^2,
\een
hence
$$|u'(r,\alpha)|\ge 4\bar\beta-2\bar u>0$$
thus $u'(r_0,\alpha)\not=0$.
 Integrating this last inequality over $(\bar r,r_0)$ and using that $u(r_0,\alpha)\ge \bar u$, we deduce
$$r_0\le C_k+\frac{1}{2}.$$
Hence $F(u(r_0,\alpha))=F(2\bar\beta)$, $u(r_0,\alpha)=\bar u$, implying
 $\alpha\in\mathcal N_1$, and by \eqref{p3},
$$E(r_0,\alpha)\ge 2(C_k+1)^NB+k\bar Q\frac{(C_k+1)^N}{N}$$
Therefore  $T_1(\alpha)<\infty$, $u(T_1(\alpha))<\bar u$ and $f(s)<0$ for $u(T_1(\alpha))\le s\le \bar u$, so there exists a first point  $r^+_0$  after $T_1(\alpha)$ at which $u$ takes the value $\bar u$. If this point is greater than   $C_k$, we are done. As $E(r_0^+,\alpha)\ge E(r_0,\alpha)$, we can repeat the above argument as many times as necessary to conclude $\alpha\in\mathcal N_k$.
\end{proof}
\medskip

\begin{proof}[{\bf Proof of Theorem \ref{main0}}]

We first observe that for each $k\in\mathbb N\cup\{0\}$, $\mathcal G_k\cup\mathcal Q_k$ is bounded by $\alpha_{k+1}$ in Lemma \ref{nn2}. We will prove next that there exists $m\in\mathbb N\cup\{0\}$ such that $\mathcal G_{m}\not=\emptyset$. Once we have done this, we shall denote by $m_1$ the first value of $m$ such that $\mathcal G_m\not=\emptyset$ and
set
$$\alpha_{m_1}^{\#}:=\inf(\mathcal G_{m_1}\cup\mathcal Q_{m_1})\quad \mbox{and}\quad\alpha_{m_1}^{*}:=\sup(\mathcal G_{m_1}\cup\mathcal Q_{m_1}).$$
Then by Lemma \ref{front}(ii) and the definition of $m_1$, $\alpha_{m_1}^{\#}, \alpha_{m_1}^{*}\in\mathcal G_{m_1}$. At this point, we cannot guarantee that $\alpha_{m_1}^{\#}< \alpha_{m_1}^{*}$. As by continuous dependence, for $\bar\alpha\in\mathcal G_{m_1}$ there is a neighborhood of $\bar\alpha$ which is contained in $\mathcal G_{m_1}\cup\mathcal Q_{m_1}\cup\mathcal N_{m_1}$. From the definition of $\alpha_{m_1}^{\#}$ and $\alpha_{m_1}^{*}$, there exists $\delta>0$ such that $(\alpha_{m_1}^{\#}-\delta,\alpha_{m_1}^{\#})\subset \mathcal N_{m_1}$ and
$(\alpha_{m_1}^{*},\alpha_{m_1}^{*}+\delta)\subset\mathcal N_{m_1}$. Hence from Lemma \ref{p2m}(i), by taking a smaller $\delta>0$ if necessary, we may assume that $(\alpha_{m_1}^{\#}-\delta,\alpha_{m_1}^{\#})\subset \mathcal Q_{m_1+1}$ and
$(\alpha_{m_1}^{*},\alpha_{m_1}^{*}+\delta)\subset\mathcal Q_{m_1+1}$. Set now
$$\alpha_{m_1+1}^{\#}=\inf(\mathcal G_{m_1+1}\cup Q_{m_1+1})\quad\mbox{and}\quad \alpha_{m_1+1}^{*}=\sup(\mathcal G_{m_1+1}\cup Q_{m_1+1}).$$
From Lemma \ref{front}(i), $\alpha_{m_1+1}^{\#}<\alpha_{m_1+1}^{*}$, and from Lemma \ref{front}(ii), $\alpha_{m_1+1}^{\#}$ and
$ \alpha_{m_1+1}^{*}$ belong to $\mathcal G_{m_1+1}$. We proceed by induction. At each step  $k\ge m_1+1$, by Lemma \ref{p2m}(i) we have that $\mathcal  Q_{k}\not=\emptyset$    so we can define
$$\alpha_{k}^{\#}=\inf(\mathcal G_{k}\cup Q_{k})\quad\mbox{and}\quad \alpha_{k}^{*}=\sup(\mathcal G_{k}\cup Q_{k})
$$
 to obtain the existence of two different elements in $\mathcal G_k$ for every $k\ge m_1+1$.
\medskip

We prove next that there exists $m\in\mathbb N\cup\{0\}$ such that $\mathcal G_{m}\not=\emptyset$.  From Lemma \ref{nn2}, set
\ben
\alpha^1:=\inf\{\alpha\ge\beta_*\ |\ (\alpha,\gamma_*)\subset \mathcal N_1\}.
\een

Then, by Lemma \ref{front}(i), either $\alpha^1\in\mathcal G_1$ or $\alpha^1\in \Upsilon_1$. In our next arguments, and when both cases are possible, we will assume the worse, that is, that the limit points that we obtain are not in $\mathcal G_k$.\\
Hence we  assume that $\alpha^1\in\Upsilon_1$. Then there exists $i\in\{1,\ldots,M\}$ such that $u(Z_1(\alpha^1),\alpha^1)=\gamma_i$. From Lemma \ref{gamma-vec} and the definition of $\alpha^1$, for any $k\in\mathbb N$,
$$\{\alpha\ge\beta_*\ |\ (\alpha^1,\alpha)\subset \mathcal N_k\}\not=\emptyset.$$
Since $\alpha^1<\gamma_*$, we can choose $d>0$ such that $\alpha^1+d<\gamma_*$ and set, for both sets of assumptions,
\ben
\alpha^1_k:=\sup\{\alpha\in(\alpha^1, \alpha^1+d)\ |\ (\alpha^1,\alpha)\subset\mathcal N_k\}.
\een
As $\{\alpha^1_k\}$ is monotone decreasing in $k$, it converges. Since \eqref{ivp} does not have oscillatory solutions, see Proposition \ref{basic}(iv), it follows that it converges to $\alpha^1$. Hence there exists $k_1>0$ such that
$$\alpha_{k_1}^1<\alpha_{k_1-1}^1<\alpha^1+d\quad\mbox{and}\quad u(Z_{k_1},\alpha_{k_1}^1)=\gamma_j,$$
with $F(\gamma_j)<F(\gamma_i)$ by Lemma \ref{p2m}(ii). We observe that by the strict inequality $\alpha_{k_1}^1<\alpha_{k_1-1}^1$, it holds that $\alpha^1_{k_1}\in\mathcal N_{k_1-1}$ and $(\alpha^1,\alpha^1_{k_1})\subset\mathcal N_{k_1}$. Set, for $k\ge k_1$,
\ben
\alpha^2_k:=\inf\{\alpha\in(\alpha^1,\alpha_{k_1}^1) \ |\ (\alpha,\alpha_{k_1}^1)\subset\mathcal N_k\}.
\een
Now the sequence $\{\alpha^2_k\}$ is monotone increasing in $k$ and the same argument yields  $\alpha_k^2\to \alpha^1_{k_1}$ as $k\to\infty$ and  there exists $k_2>k_1$ such that $$\alpha^2_{k_2-1}<\alpha^2_{k_2}$$
so that $\alpha^2_{k_2}\in\mathcal N_{k_2-1}$, $(\alpha_{k_2}^2,\alpha^1_{k_1})\subset\mathcal N_{k_2}$, and $u(Z_{k_2},\alpha_{k_2}^2)=\gamma_\ell$ with $F(\gamma_\ell)<F(\gamma_j)$, again by Lemma \ref{p2m}(ii). We may continue in this way by setting, for $k\ge k_2$,
\ben
\alpha^3_k:=\sup\{\alpha\in(\alpha_{k_2}^2,\alpha_{k_1}^1) \ |\ (\alpha_{k_2}^2,\alpha)\subset\mathcal N_k\}.
\een
After a finite number of steps we will reach $\gamma_0$ obtaining an $\alpha\in\mathcal G_{k_m}$ for some $k_m\in\mathbb N$.
\end{proof}

\begin{proof}[{\bf Proof of Theorem \ref{main1}}]

\noindent We prove it first for the case that $f$ satisfies $(A1)$.  Assume by contradiction that there exists  $\alpha>\beta_*$ in $\mathcal G_j$, that is $u(\cdot,\alpha)$ has $j-1$ sign changes, for some $j=1,\ldots,k+1$. As $u(\cdot,\alpha)$ crosses the value $\gamma_1$ at a first point $r_{\gamma_1}^1$,  from $|u'(r)|\le (2(F(\gamma_*)+\tilde F))^{1/2}$ for all $r\le r_{\gamma_1}^1$, we find that
$$r_{\gamma_1}^1\ge \frac{\beta_*-\gamma_1}{(2(F(\gamma_*)+\tilde F))^{1/2}},$$
where $\tilde F$ is defined in Lemma \ref{gamma-vec}.
Let $r_{\gamma_1}\ge r_{\gamma_1}^1 $ denote the last point at which $F(u(r_{\gamma_1}))=F(\gamma_1)$, and we may assume it happens after $T_i$, for some $0\le i<j$. Using that $I(Z_j)=0$, we find that
\ben
I(r_{\gamma_1})&=&(N-1)\int_{r_{\gamma_1}}^{Z_j}\frac{|u'|^2}{r}dr\\
&\le&\frac{N-1}{r_{\gamma_1}}(2(F(\gamma_*)+\tilde F))^{1/2}\int_{r_{\gamma_1}}^{Z_j}|u'(r)|dr\\
&=&\frac{N-1}{r_{\gamma_1}}(2(F(\gamma_*)+\tilde F))^{1/2}\Bigl(\int_{r_{\gamma_1}}^{T_{i+1}}|u'(r)|dr+\int_{T_{i+1}}^{T_{i+2}}|u'(r)|dr+\cdots\int_{T_{j-1}}^{Z_j}|u'(r)|dr\Bigr)\\
&\le& \frac{(N-1)2(F(\gamma_*)+\tilde F)}{\beta_*-\gamma_1}(j-i)(\gamma_1-\bar u),
\een
we find that
$$F(\gamma_1)\le(N-1)2(F(\gamma_*)+\tilde F)\frac{j(\gamma_1-\bar u)}{\beta_*-\gamma_1},$$
a contradiction to \eqref{t13}.

\noindent If $f$ satisfies $(A2)$, we let $\alpha_k$ be as defined in Lemma \ref{nn2}. Then we only have to prove that there cannot exist solutions to \eqref{eq2} with initial value $\alpha<\alpha_k$. But then, as
 $|u'(r)|\le (2(\sup_{s\in[0,\alpha_k]}F(s)+\tilde F))^{1/2}$ for all $r>0$, we can argue as  above to obtain  contradiction to \eqref{t13b}.
\end{proof}

\medskip
In order to prove our last result, we need the following lemma, which is another generalization of
 \cite[Lemma 3.1]{gst}.
\begin{lemma}\label{gstlema}

Let $f$ satisfy either $(A1)$or $(A2)$,   $\bar\beta$ be as in Definition \ref{constants}(iii), and  $\alpha>\bar\beta$.
 Let $ r_{\bar\beta}$ be the first point  at which $u( \rbarbetad,\alpha)=\bar\beta$. If
\ben r_{\bar\beta}\ge \bar C:=2(N-1)\frac{\bar\beta-\beta_1}{F(\bar\beta)-F(\gamma_M)}(2(F(\bar\beta)+\hat F))^{1/2},
\een
where $\hat F:= -\min_{s\in[\beta_1,\beta_{*}]}F(s)$, then there exists a first point $r_{\beta_1}>r_{\bar\beta}$ such that $u(r_{\beta_1},\alpha)=\beta_1$, $u'(r_{\beta_1},\alpha)<0$, and
$$r_{\beta_1}-r_{\bar\beta}\le \frac{\beta_*-\beta_1}{(F(\bar\beta)-F(\gamma_M))^{1/2}}+\Bigl(\frac{2 N(\bar\beta-\beta_*)}{\min\limits_{t\in[\beta_*,\bar\beta]}f(t)}\Bigr)^{1/2}.$$

\end{lemma}
\begin{proof}
  As any solution satisfying $\alpha>\beta_*$ must cross $\beta_*$ at a first point that we denote by $r_{\beta_*}$, we integrate \eqref{Ider} over $[\rbarbetad , r]$ with $r>r_{\beta_*}$, and obtain

$$I(r)=I(\rbarbetad)-(N-1)\int_{r_{\bar\beta}}^r\frac{|u'(t)|^2}{t}dt.$$
Since $|u'(r)|\le 2^{1/2}(I(r_{\bar\beta})+\hat F)^{1/2}$ as long as $u(r)\ge\beta_1$, we find that
$$I(r)\ge  I(r_{\bar\beta})\Bigl(1-\frac{\sqrt{2(I(r_{\bar\beta})+\hat F)}}{I(r_{\bar\beta})}\frac{(N-1)(\bar\beta-u(r))}{r_{\bar\beta}}\Bigr).$$
As as
$\sqrt{2(I+\tilde F)}/I$ is decreasing in $I$, $I(r_{\bar\beta})\ge F(\bar\beta)$, and $r_{\bar\beta}\ge \bar C$, we find that
\ben
I(r)&\ge&  I(r_{\bar\beta})\Bigl(1-\frac{F(\bar\beta)-F(\gamma_M)}{2F(\bar\beta)}\Bigr)\ge\frac{F(\bar\beta)+F(\gamma_M)}{2}\\
&\ge & F(u(r))+\frac{F(\bar\beta)-F(\gamma_M)}{2},
\een
implying that as long as $\beta_*\ge u(r)\ge \beta_1$,
$$|u'(r)|\ge \bigl(F(\bar\beta)-F(\gamma_M)\bigr)^{1/2}.$$
Hence $u(Z_{1})<\beta_1$, and
$$\beta_*-\beta_1\ge \bigl(F(\bar\beta)-F(\gamma_M)\bigr)^{1/2}(r_{\beta_1}-r_{\beta_*}).$$
Finally,  by integrating the equation in \eqref{ivp} over $[\rbarbetad,r]$ with $r\le r_{\beta_*}$ we find that
$$r^{N-1}|u'(r)| =\int_{\rbarbetad}^rt^{N-1}f(u(t))dt\ge\min_{s\in[\beta_*,\bar\beta]}f(s)\frac{r^N-\rbarbetad^N}{N},$$
hence
$$|u'(r)|\ge \Bigl(\frac{\min_{s\in[\beta_*,\bar\beta]}f(s)}{N}\Bigr)^{ }(r-\rbarbetad)^{ },$$
implying that
$$2(\bar\beta-\beta_*)\ge \Bigl(\frac{\min_{s\in[\beta_*,\bar\beta]}f(s)}{N}\Bigr)^{ }(r_{\beta_*}-\rbarbetad)^{2},$$
hence the result follows.
\end{proof}

\begin{proof}[{\bf Proof of Theorem \ref{main2}}]
To prove this theorem we only need to prove than under its assumptions, $\mathcal Q_1\not=\emptyset$. Setting
$$\alpha_{1}^{\#}:=\inf(\mathcal G_{1}\cup\mathcal Q_{1})\quad \mbox{and}\quad\alpha_{1}^{*}:=\sup(\mathcal G_{1}\cup\mathcal Q_{1}),$$
and observing that from Lemma \eqref{front}(i) $\alpha_{1}^{\#}<\alpha_{1}^{*}$ we may argue as in the proof of Theorem \ref{main0} to obtain the desired result.\\

Let $\alpha^*>\beta_*$ be such that $u(\cdot,\alpha^*)=u(\cdot)$ crosses the value $\beta_1$. For simplicity of notation we will set $Z_1=Z_1(\alpha^*)$,  $I(Z_1)=I(Z_1,\alpha^*)$ and $I(r_{\beta_1})=I(r_{\beta_1},\alpha^*)$. As $|u'(r)|\le \sqrt{2(I(r_{\beta_1})+\bar F)}$ for $r\in (r_{\beta_1},Z_1)$, by integrating \eqref{Ider} we have
\ben
Z_1^{2(N-1)}I(Z_1)&=&r_{\beta_1}^{2(N-1)}I(r_{\beta_1})-2(N-1)\int_{r_{\beta_1}}^{Z_1}r^{2N-3}|F(u(r))|dr\\
&\le &r_{\beta_1}^{2(N-1)}I(r_{\beta_1})-2(N-1)r_{\beta_1}^{2N-3}\int_{r_{\beta_1}}^{Z_1}|F(u(r))|dr\\
&\le &r_{\beta_1}^{2(N-1)}I(r_{\beta_1})-\frac{2(N-1)r_{\beta_1}^{2N-3}}{(2(I(r_{\beta_1})+\bar F))^{1/2}}\int_{r_{\beta_1}}^{Z_1}|F(u(r))u'(r)|dr\\
&\le&r_{\beta_1}^{2N-3}\Bigl(r_{\beta_1}I(r_{\beta_1})-\frac{2(N-1)}{(2(I(r_{\beta_1})+\bar F))^{1/2}}\int_0^{\beta_1}|F(s)|ds\Bigr).
\een
Hence, if
\beq\label{need}
r_{\beta_1}I(r_{\beta_1})<\frac{2(N-1)}{(2(I(r_{\beta_1})+\bar F))^{1/2}}\int_0^{\beta_1}|F(s)|ds,
\eeq
then $\alpha^*\in\mathcal Q_1$. The proof of this theorem consists in finding  an $\alpha^*$ such that $u(\cdot,\alpha^*)$ crosses $\beta_1$ and  \eqref{need} holds
\medskip

In what follows,  $ \bar C$ is as in Lemma \ref{gstlema}. \\
\medskip

If $\gamma_*<\infty$, and as $f(\gamma_*)=0$, by continuous dependence of the solution of \eqref{ivp} in the initial data, we have that $r_{\bar\beta}(\alpha)\to \infty$ as $\alpha\to\gamma_*$, hence we can choose $\alpha^*>\beta_*$ so that $r_{\bar\beta}(\alpha^*)= \bar C$. Using now that $I(r)\le F(\gamma_*)$, we see that from \eqref{caso1}, $\alpha^*\in\mathcal Q_1$.\\
\medskip

 Let now $\gamma_*=\infty$.
and set
$$M:=( \bar C+1)^N\Bigl(2F(\bar\beta)+(\bar\beta-\beta_1)^2+\frac{\bar Q}{N}\Bigr),$$
where $\bar Q:=-\min_{s\in[s_0,\bar\beta]}Q(s)\ge0$. Using the same argument used in the proof of Lemma \ref{nn2}, we have that
$$\lim_{\alpha\to\infty}E(r_{\bar\beta}(\alpha),\alpha)=\infty.$$
Since $E(r_{\bar\beta}(\bar\beta),\bar\beta)=0$, by continuity there exists a smallest $\bar\alpha>\bar\beta$ such that $E(r_{\bar\beta}(\bar\alpha),\bar\alpha)=M$.

If $\rbarbetad(\bar\alpha)\ge  \bar C$, then again by continuity we can choose $\alpha^*\le\bar\alpha$ such that $\rbarbetad(\alpha^*)=  \bar C$. Moreover, since $E(\rbarbetad(\alpha^*),\alpha^*)\le M$, we find that
$$I(\rbarbetad(\alpha^*))-\frac{(N-2)\bar\beta }{2^{1/2} \bar C}(I(\rbarbetad(\alpha^*)))^{1/2}\le \frac{M}{2 \bar C^N}$$
and hence
\ben
I(\rbarbetad(\alpha^*))\le \frac{(N-2)^2\bar\beta^2}{2 \bar C^2}+\frac{M}{ \bar C^N}.
\een
hence, using $I(r_{\beta_1})\le I(r_{\bar\beta})$ and assumption \eqref{caso1} we obtain that \eqref{need} holds
and thus $\alpha^*\in\mathcal Q_1$.

\medskip

Let now $r_{\bar\beta}(\bar\alpha)<  \bar C$. We will first prove that in this case $u=u(\cdot,\bar\alpha)$ crosses the value $\beta_1$ and $r_{\beta_1}(\bar\alpha)<  \bar C+1$.

Let $r_0=r_0(\bar\alpha)$ be the first point after $\bar r(\bar\alpha)$ such that either
$$r_0=\bar C+1,\quad\mbox{or}\quad u'(r_0,\bar\alpha)=0,\quad\mbox{or}\quad u(r_0,\bar\alpha))=\beta_1.$$
Integrating \eqref{energy} over $[0,r]$ with $r\le r_0$ we get
\ben
2r^NI(r)&\ge&\int_0^{r}t^{N-1}Q(u(t))dt\\
&\ge& M-\frac{\bar Q}{N}r^N\\
&=& ( \bar C+1)^N\Bigl(2F(\bar\beta)+(\bar\beta-\beta_1)^2+\frac{\bar Q}{N}\Bigr)-\frac{\bar Q}{N}r^N\\
&\ge&( \bar C+1)^N\Bigl(2F(\bar\beta)+(\bar\beta-\beta_1)^2\Bigr)
\een
and therefore
$$F(\bar\beta)+\frac{|u'(r)|^2}{2}\ge I(r)\ge F(\bar\beta)+\frac{1}{2}(\bar\beta-\beta_1)^2.$$
We conclude then that $|u'(r)|\ge \bar\beta-\beta_1$ and thus $u'(r_0)\not=0$.
Integrating this last inequality over $[ r_{\bar\beta},r_0]$ we deduce that $u(r_0)<\bar C+1$. Hence, $u(r_0)=\beta_1$.

\medskip

We conclude that
$$\alpha^*:=\inf\{\alpha>\bar\beta\ |\ r_{\beta_1}(s)< \bar C+1\quad\mbox{for all $s\in(\alpha,\bar\alpha)$}\}$$
is well defined. We will show that $u(\cdot,\alpha^*)$ crosses the value $\beta_1$.
If not, then $\alpha^*\in\mathcal S_1\cup \Upsilon_1$, and as $\mathcal S_1$ is open, it must be that $\alpha^*\in\Upsilon_1$. But then $Z_1(\alpha^*)=\infty$, and $u( \bar C+1,\alpha^*)>\gamma_1$, hence by continuity we obtain a contradiction.

 If $\alpha^*=\bar\beta$, then by using that $I(r,\alpha)\le F(\alpha)$ for all $\alpha$, we find that
$$r_{\beta_1}(\bar\beta)I(r_{\beta_1},\bar\beta)\le ( \bar C+1)F(\bar\beta)$$
and hence by assumption \eqref{caso1} again \eqref{need} holds implying $\bar\beta\in\mathcal Q_1$.

If $\alpha^*>\bar\beta$, then it must be that $r_{\beta_1}(\alpha^*)= \bar C+1$.
Hence, as
 $$E(r_{\beta_1})=E(\rbarbetad)+\int_{\rbarbetad}^{r_{\beta_1}}t^{N-1}Q(u(t))dt<M+\sup_{s\in[\beta_1,\bar\beta]}Q(s)\frac{r_{\beta_1}^N}{N},$$
we find that
$$I( \bar C+1)-\frac{(N-2)\beta_1}{( \bar C+1)}(I( \bar C+1))^{1/2}<F(\bar\beta)+\frac{1}{2}(\bar\beta-\beta_1)^2+\frac{\bar Q}{2N}+\frac{1}{2N}\sup_{s\in[\beta_1,\bar\beta]}Q(s)$$
and thus
$$I( \bar C+1)\le 2F(\bar\beta)+(\bar\beta-\beta_1)^2+\frac{\bar Q}{N}+\frac{1}{N}\sup_{s\in[\beta_1,\bar\beta]}Q(s)+\Bigl(\frac{(N-2)\beta_1}{ \bar C+1}\Bigr)^2.$$

Hence, by assumption \eqref{caso1} we have that \eqref{need} holds and thus $\alpha^*\in\mathcal Q_1$.

\end{proof}

\section{Appendix}
In this section we prove that solutions  to \eqref{ivp} cannot be oscillatory. This was done in \cite{cdghm} under different assumptions on $f$ but its proof can be adapted to the present case without any difficulty. We include it here for the sake of completeness.

\begin{proof}[Proof of Lemma \ref{basic}(iv)]

We argue by contradiction and suppose that there is an infinite sequence $\{z_n\}$ (tending to infinity) of simple zeros of $u$.
We denote by $\{z_n^+\}$ the zeros for which $u'(z_n^+)>0$ and by $\{z_n^-\}$ the zeros for which $u'(z_n^-)<0$. We have
\[
0<z_1^-<z_1^+<z_2^-<\cdots<z_n^+<z_{n+1}^-<z_{n+1}^+<\cdots
\]
Between $z_n^-$ and $z_n^+$ there is a minimum $r_n^m$ where $u(r_n^m)<0$ and between $z_n^+$ and $z_{n+1}^-$ there is a maximum $r_n^M$ where $u(r_n^M)>0$.
By Proposition \ref{basic}(ii), $F(u(r_n^M)),\ F(u(r_n^m))\to F(\ell)$ where $\ell$ is a zero of $f$ and $F(\ell)\ge 0$. Let $\mu^-,\ \mu^+$ be the unique points $\mu^-<0<\mu^+$ such that $f(\mu^-)\not=0$, $f(\mu^+)\not=0$, $F(\mu^-)=F(\ell)=F(\mu^+)$ and $F(s)\le F(\ell)$ for all $s\in(\mu^-,\mu^+)$. Let $\{u(r_{k_n}^M)\}$  be any convergent subsequence of $\{u(r_n^M)\}$ and let $\bar u$ be its limit. Then $F(\bar u)=F(\ell)$. As $u$ is oscillatory, we must have that for each $n$, $F(s)\le F(u(r_{k_n}^M))$ for all $u(r_{k_n+1}^m)\le s\le u(r_{k_n}^M)$. In particular, $\bar u$ cannot be a local minimum of $F$. By Lemma \ref{p2m}(ii), we have that $\bar u$ cannot be a local maximum of $F$ either. As $F(s)\le F(\bar u)$ for all $0\le s\le \bar u$, $\bar u=\mu^+$. Using the same argument, any other convergent subsequence of $\{u(r_{n}^M)\}$ has to converge to $\mu^+$. Similarly, $u(r_n^m)$ converges to $\mu^-$.

As both $I(r_n^M)$ and $I(r_n^m)$ are greater than or equal to $F(\ell)$, it must be that $u(r_n^m)<\mu^-$ and $u(r_n^M)>\mu^+$.

As $f(\mu^+)>0$, there exists $\nu>0$ such that $f(s)>0$ for all $s\in[\mu^+-\nu,\mu^++\nu]$ and $f(s)<0$ in $[\mu^--\nu,\mu^-+\nu]$
and we set
$$\bar f:=\min_{s\in [\mu^+-\nu,\mu^++\nu]}f(s)\quad\mbox{and}\quad \bar{\bar f}:=\max_{s\in [\mu^--\nu,\mu^-+\nu]}f(s).$$
We define next the unique points
$$r_{1,n}\in(z_n^+,r_n^M),\quad r_{2,n}\in(r_n^M,z_{n+1}^-),\quad s_{1,n},\ \bar s_{1,n}\in(r_{2,n}, z_{n+1}^-),\quad t_{1,n}\in(z_{n+1}^-, r_{n+1}^m),$$
so that
\[
u(r_{1,n})=\mu^+-\nu=u(r_{2,n})\,,\quad u(s_{1,n})=\delta/2\,,\quad u(\bar s_{1,n})=\delta/4,\quad u(t_{1,n})=\mu^-+\nu\,.
\]
where $\delta$ is defined in $(f_2)$. We have
\[
z_n^+<r_{1,n}<r_n^M<r_{2,n}<s_{1,n}<\bar s_{1,n} <z_{n+1}^-<t_{1,n}<r_{n+1}^m\,.
\]
For $r\in(r_{2,n},t_{1,n})$, $\mu^-<u(r)<\mu^+$, hence $F(u(r))\le F(\ell)$. Also, for $r\in(s_{1,n},\bar s_{1,n})$, $|F(u(r))-F(\ell)|\ge k_0$ for some positive constant $k_0$ independent of $n$. Moreover, by applying the mean value theorem, and Proposition \ref{basic}, we get that there exists a constant $k_1$, which is independent of $n$, such that
\[
0<k_1\le \bar s_{1,n}-s_{1,n}.
\]
From~\eqref{ivp} we have that
\[
|u''(r)|\;=\;\Bigm|\frac{N-1}r\,u'(r)+f(u(r))\Bigm|\;\ge\;\bar f-\frac{N-1}r\,C(\alpha)
\]
for any $r\in[r_{1,n},r_{2,n}]$. If additionally $r\ge\bar  r:=2\,(N-1)\,C(\alpha)/\,\bar f$, then the r.h.s.~in the above inequality is bounded from below by $\bar f/2$. Hence, choosing $n_0$ such that $z_n^+\ge\bar  r$ for all $n\ge n_0$, we have that
\[
|u''(r)|\ge \frac12\,\bar  f\quad\mbox{for all }r\in[r_{1,n},r_{2,n}]
\]
and therefore, again from the mean value theorem and Proposition \ref{basic}, we get that
\[
2\,C(\alpha)\ge |u'(r_{2,n})-u'(r_{1,n})|=|u''(\xi)|\,(r_{2,n}-r_{1,n})\ge\frac12\,\bar  f\,(r_{2,n}-r_{1,n})
\]
implying that
\[\label{nos1}
r_{2,n}-r_{1,n}\le\frac{2\,C(\alpha)}{\bar  f}\;.
\]
Let $\tilde H$ be as in \eqref{tildeH} with $\gamma$ replaced by $\ell$, that is,
$$\tilde H(r,\alpha)=r^{2(N-1)}(I(r,\alpha)-F(\ell)),$$
so that
$$\tilde H'(r,\alpha)=2(N-1)r^{2N-3}(F(u(r,\alpha))-F(\ell)).$$
We have
\begin{eqnarray}
\frac{\tilde H(t_{1,n})-\tilde H(r_{1,n})}{2(N-1)}&=&\int_{r_{1,n}}^{r_{2,n}}r^{2N-3}\,(F(u)-F(\ell))\;dr+
\int_{r_{2,n}}^{t_{1,n}}r^{2N-3}\,(F(u)-F(\ell))\;dr\nonumber\\
&=&\int_{r_{1,n}}^{r_{2,n}}r^{2N-3}\,(F(u)-F(\ell))\;dr-
\int_{r_{2,n}}^{t_{1,n}}r^{2N-3}\,|F(u)-F(\ell)|\;dr\nonumber\\
&\le& \int_{r_{1,n}}^{r_{2,n}}r^{2N-3}\,(F(u)-F(\ell))\;dr-
\int_{s_{1,n}}^{\bar s_{1,n}}r^{2N-3}\,|F(u)-F(\ell)|\;dr\nonumber\\
&\le& (F(u(r_n^M))-F(\ell))r_{2,n}^{2N-3}\frac{2\,C(\alpha)}{\bar  f}-\,k_0\,k_1\,r_{2,n}^{2N-3}\,.\nonumber
\end{eqnarray}
Since $\lim_{n\to+\infty}F(u(r_n^M))-F(\ell)=0$, we can choose $n_0$ large enough so that
\[
\frac{2\,C(\alpha)}{\bar  f}\,(F(u(r_n^M))-F(\ell))-\,k_0\,k_1<-\frac{1}{2}\,k_0\,k_1
\]
for all $n\ge n_0$, and hence
\[
\tilde H(t_{1,n})-\tilde H(r_{1,n})\le-(N-1)\,k_0\,k_1\,r_{2,n}^{2N-3}\,.
\]
Clearly, we can repeat the above argument in the interval $(t_{1,n},r_{1,n+1})$, thus proving that
\[
\tilde H(r_{1,n_0+j})-\tilde H(r_{1,n_0})\le-(N-1)\,k_0\,k_1\sum_{i=0}^{j-1}\Bigl(r_{2,n_0+i}^{2N-3}+t_{2,n_0+i}^{2N-3}\Bigr)
\]
where $t_{2,n}\in(r_{n+1}^m,z_{n+1}^+)$ is uniquely defined by the condition $u(t_{2,n})=\mu^-+\nu$. Hence
\[
\lim_{j\to+\infty}\tilde H(r_{1,n_0+j})=-\infty\,,
\]
implying the contradiction that $I(r_{1,n_0+j})<F(\ell)$ for some $j$ large enough.
\end{proof}

\end{document}